\documentclass[12pt]{amsart}

\usepackage{amssymb}
\usepackage{mathrsfs}
\usepackage{enumerate}

\newcommand\Cfrak{\mathfrak{C}}
\newcommand\Rbb{\mathbb{R}}
\newcommand\Dcal{\mathcal{D}}
\newcommand\axiom{\mathrm}
\newcommand\MA{\axiom{MA}}
\newcommand{\Ht}{\operatorname{ht}}
\newcommand{\Osc}{\operatorname{Osc}}
\newcommand{\osc}{\operatorname{osc}}
\newcommand\Dbb{\mathbb{D}}
\newcommand\rbf{\mathbf{r}}
\newcommand{\arrow}[1]{\overrightarrow{#1}}
\newcommand{\ind}{\operatorname{ind}}
\newcommand{\otp}{\operatorname{otp}}
\newcommand{\bigmeet}{\bigwedge}
\newcommand\Pscr{\mathscr{P}}
\newcommand{\Seq}[1]{\langle #1 \rangle}
\newcommand{\symdif}{\triangle}

\newcounter{my_enumerate_counter}
\newcommand{\pushcounter}{\setcounter{my_enumerate_counter}{\value{enumi}}}
\newcommand{\popcounter}{\setcounter{enumi}{\value{my_enumerate_counter}}}

\newtheorem{thm}{Theorem}[section]
\newtheorem{lem}[thm]{Lemma}

\title[Forcing Axioms and CH, part II]
{Forcing Axioms and the Continuum Hypothesis, part II:
Transcending $\omega_1$-sequences of real numbers}



\thanks{
The research of the author presented in this paper was supported by
NSF grant DMS--0757507.
Any opinions, findings, and conclusions or recommendations expressed in this article
are those of the author
and do not necessarily reflect the views of the National Science Foundation. 
}

\author{Justin Tatch Moore}

\begin{document}

\begin{abstract}
The purpose of this article is to prove that the forcing axiom for completely proper forcings is
inconsistent with the Continuum Hypothesis.
This answers a longstanding problem of Shelah.
The corresponding completely proper forcing which can be constructed using CH
is moreover a tree whose square is \emph{special off the diagonal}.
While such trees had previously been constructed by Jensen and Kunen under the assumption
of $\diamondsuit$, this is the first time such a construction has been carried out using the Continuum Hypothesis.
\end{abstract}
\maketitle

\section{Introduction}
In his seminal analysis of the cardinality of infinite sets, Cantor demonstrated a procedure
which, given an $\omega$ sequence of real numbers, produces a new real number which is
not in the range of the sequence.
He then posed the Continuum Problem, asking whether the cardinality of the real line was $\aleph_1$
or some larger cardinality.
This problem was eventually resolved by the work of G\"odel and Cohen who proved that
the equality $|\Rbb| = \aleph_1$ was neither refutable nor provable,
respectively, within the framework of ZFC.
In particular, there is no procedure in ZFC which
takes an $\omega_1$-sequence of real numbers and produces a new real number not in the range of the
original sequence.

The purpose of this article is to revisit this type of diagonalization and to demonstrate that
models of the Continuum Hypothesis exhibit a strong form of incompactness.
The primary motivation for the results in this paper is to answer a question
of Shelah \cite[2.18]{problems:Shelah}
concerning which \emph{forcing axioms} are consistent with the Continuum Hypothesis (CH).
The forcing axiom for a class $\Cfrak$ of partial orders is the assertion that
if $\Dcal$ is a collection of $\aleph_1$ cofinal subsets of a partial order, then there
is an upward directed set $G$ which intersects each element of $\Dcal$.
The well known \emph{Martin's Axiom for $\aleph_1$ dense sets} ($\MA_{\aleph_1}$) is the
forcing axiom for the class of \emph{c.c.c.} partial orders.
Foreman, Magidor, and Shelah have isolated the largest class of partial orders
$\Cfrak$ for which a forcing axiom is consistent with ZFC.
Shelah has established certain sufficient conditions on a class $\Cfrak$ of partial orders in
order for the forcing axiom for $\Cfrak$ to be consistent with CH.
His question concerns to what extent his result was sharp.
The main result of the present article, especially when combined with those of
\cite{FA_CH} show that this is essentially the case.

The solution to Shelah's problem has some features
which are of independent interest.
Suppose that $T$ is a collection of countable closed subsets of $\omega_1$
which is closed under taking
closed initial segments and that $T$ has
the following property:
\begin{enumerate}

\item \label{ad_cond}
Whenever $s$ and $t$ are two elements of $T$ with the same supremum $\delta$
and $\lim(s) \cap \lim(t)$ is unbounded in $\delta$, then $s = t$.

\pushcounter
\end{enumerate}
(Here $\lim (s)$ is the set of limit points of $s$.)
As usual, we regard $T$ as a set-theoretic tree
by declaring $s \leq t$ if $s$ is an initial part of $t$.
Observe that this condition implies that $T$ contains at most one uncountable path:
any uncountable path would have a union which is a closed unbounded subset of $\omega_1$
and such subsets of $\omega_1$ must have an uncountable intersection.
In fact it can be shown that this condition implies that
\[
\{(s,t) \in T^2 : (\Ht_T(s) = \Ht_T(t)) \land (s \ne t)\}
\]
can be decomposed into countably many antichains
($T^2$ is \emph{special off the diagonal}).

The main result of this article is to prove that if the Continuum Hypothesis is true,
then there is tree $T$ which satisfies (\ref{ad_cond}) together with the following additional properties:
\begin{enumerate}
\popcounter

\item \label{no_path}
$T$ has no uncountable path;

\item \label{density}
for each $t$ in $T$, there is a closed an unbounded set of $\delta$ such that
$t \cup \{\delta\}$ is in $T$;

\item \label{completely_proper}
$T$ is proper as a forcing notion and remains so in any outer model
with the same set of real numbers in which $T$ has no uncountable path.
Moreover $T$ is complete with respect
to a simple $\aleph_1$-completeness system $\Dbb$ and in fact is $\aleph_1$-completely
proper in the sense of \cite{minimal_unctbl_types}.

\item \label{effective_bounding}
If $X$ is a countable subset of $T$, then the collection of bounded chains
which are contained in $X$ are Borel as a subset of $\Pscr (X)$.

\pushcounter
\end{enumerate}
This provides the first example of a consequence of the Continuum Hypothesis which justifies
condition (\ref{two_alternates}) in the next theorem.
\begin{thm} \cite{proper_forcing} \cite{1st_ctbl_ctbly_cpt}
Suppose that $\Seq{P_\alpha; Q_\alpha : \alpha \in \theta}$ is a countable
support iteration of proper forcings which:
\begin{enumerate}[(a)]

\item \label{completeness_cond}
are complete with respect to a simple $2$-completeness system $\Dbb$;

\item \label{two_alternates} satisfy either of the following conditions:
\begin{enumerate}[(i)]

\item \label{ap} are weakly $\alpha$-proper for every $\alpha \in \omega_1$;

\item are proper in every proper forcing extension with the same set of real numbers.

\end{enumerate}
\end{enumerate}
Then forcing with $P_\theta$ does not introduce new real numbers.
\end{thm}
Results of Devlin and Shelah \cite{weak_diamond} have long been known to imply
that condition (\ref{completeness_cond}) is necessary in this theorem.
The degree of necessity of condition (\ref{two_alternates}) --- and (\ref{ap}) in particular ---
has long been a source of mystery, however.
Shelah has already shown \cite[XVIII.1.1]{proper_forcing}
that there is an iteration of forcings in $L$
which satisfies (\ref{completeness_cond}) and which introduces a new real at limit stage $\omega^2$.
This construction, however, has two serious limitations:
it is not possible in the presence of modest large cardinal hypotheses --- such as the existence of
a measurable cardinal --- and it does not refute the consistency of the corresponding
\emph{forcing axiom} with the Continuum Hypothesis.
Still, this construction is starting point for the results in the present article.

Properties (\ref{density}) and (\ref{completely_proper}) of $T$
imply that forcing with $T$ adds
an uncountable path through $T$ and does not introduce new real numbers.
Thus $T$ can have an uncountable branch in an outer model with the same real numbers.
By condition (\ref{ad_cond}),
this branch is in a sense \emph{unique}, in that once such an uncountable branch exists,
there can never be another.
(Of course different forcing extensions will have different cofinal branches through $T$ and
in any outer model in which $\omega_1^V$ is countable, there are a continuum of paths through $T$
whose union is cofinal in $\omega_1^V$.)

Property (\ref{effective_bounding}) of the tree $T$ is notable because it causes a number
of related but formally different formulations of \emph{completeness} to be identified \cite{minimal_unctbl_types}.
This is significant because it was demonstrated in \cite{FA_CH}
that in general different notions of completeness can give rise to incompatible iteration
theorems and incompatible forcing axioms.

The existence of the tree $T$ can be seen as the obstruction to a diagonalization procedure
for $\omega_1$-sequences of reals.
For each $\omega_1$-sequence of reals $\rbf$, there will be an associated sequence
$\vec{T}^\rbf$ of trees of length at most $\omega^2$
which satisfy (\ref{ad_cond}).
These sequences are such that they have positive length exactly when $(\omega_1)^{L[\rbf]} = (\omega_1)^V$
and if $\xi$ is less than the length of the sequence, then:
\begin{enumerate}
\popcounter

\item $T_\xi^{\rbf}$ has an uncountable path if and only if it is not the last
entry in the sequence;

\item if $\xi' \in \xi$, then every element of $T_{\xi}$ is contained in the set of limit points of
$E_{\xi'}$, where $E_{\xi'}$ is the union of the uncountable path through $T^{\rbf}_{\xi'}$;

\item if $T_{\xi}^{\rbf}$ is the last entry in the sequence,
then either $L[\rbf]$ contains an real number not in
the range of $\rbf$ or else $T_\xi^{\rbf}$ is completely proper in every outer model with the same real numbers.

\item if the length of the sequence is $\omega^2$ and $\delta$ is in the intersection of
$\bigcap_{\xi \in \omega^2} E_\xi$, then 
$\Seq{E_\xi \cap \delta : \xi \in \omega^2}$ is not in $L[\rbf]$.
(In particular, $\rbf$ is not an enumeration of $\Rbb$.)

\pushcounter
\end{enumerate}
Moreover the (partial) function $\rbf \mapsto \vec{T}^\rbf_\xi$ is $\Sigma_1$-definable
for each $\xi \in \omega^2$.
Thus in any outer model, the sequence $\vec{T}^\rbf$ may increase in length, but it maintains the entries from the
inner model.
Observe that if $\rbf$ is an enumeration of $\Rbb$ in ordertype $\omega_1$,
then $\vec{T}^{\rbf}$ has successor length $\eta_{\rbf} +1$.

Shelah has proved (unpublished) that an iteration of completely proper forcings of length less
than $\omega^2$ does not add new real numbers.
Thus if $\rbf$ is a well ordering of $\Rbb$ in type $\omega_1$ and $\xi \in \omega^2$,
then it is possible to go into a forcing extension with the same reals such that
$\xi \leq \eta_{\rbf}$.
The above remarks show that this result is optimal.
Also, for a fixed $\xi \in \omega^2$,
the assertion that there is an enumeration $\rbf$ of $\Rbb$ in type $\omega_1$ with
$\eta_{\rbf}$ at least $\xi$ is expressible by a $\Sigma^2_1$-formula.
By Woodin's $\Sigma^2_1$-Absoluteness Theorem (see \cite{stationary_tower}),
this means that, in the presence of a measurable
Woodin cardinal, CH implies that for each $\xi \in \omega^2$,
there is an enumeration $\rbf$ of $\Rbb$ in type $\omega_1$ such
that $\eta_{\rbf}$ is at least $\xi$.

While the construction of the sequence of trees
is elementary, the reader is assumed to have a solid background in set theory at the
level of \cite{set_theory:Kunen}.
Knowledge of proper forcing will be required at some points, although the necessary background
will be reviewed for the readers convenience;
further reading can be found in \cite{1st_ctbl_ctbly_cpt} \cite{tutorial_NNR} \cite{proper_forcing}.
Notation is standard and will generally follow that of \cite{set_theory:Kunen}.

\section{Background on complete properness}

In this section we will review some definitions associated to proper forcing, culminating in the definition of
complete properness.
This is not necessary to understand the main construction; it is only necessary in order
to understand the verification of (\ref{completely_proper}).
A \emph{forcing notion} is a partial order with a least element.
Elements of a forcing notion are often referred to as \emph{conditions}.
If $p$ and $q$ are conditions in a forcing, then $p \leq q$ is often read ``$q$ extends
$p$'' or ``$q$ is stronger than $p$.''
If to conditions have a common extension, they are \emph{compatible};
otherwise they are \emph{incompatible}.

Let $H(\theta)$ denote the collection of all sets of hereditary cardinality at most $\theta$.
If $Q$ is a forcing notion in $H(\theta)$, then
a countable elementary submodel $M$ of $H(\theta)$ is \emph{suitable} for $Q$ if
it contains the powerset of $Q$.
If $M$ is a suitable model for $Q$ and $\bar q$ is in $Q$, then we say that
$\bar q$ is \emph{$(M,Q)$-generic} if whenever $\bar q \leq r$ and $D \subseteq Q$ is
a dense subset of $Q$ in $M$, $r$ is compatible with some element of $D \cap M$.
%It is not difficult to show that a condition is generic for a model $M$ if and only if
%it forces that $\dot G \cap \check M$ is an $(M,Q)$-generic filter where $\dot G$ is the $Q$-name for
%the generic filter.
A forcing notion $Q$ is \emph{proper} whenever $M$ is suitable for $M$ and $q$ is in $Q \cap M$,
there is an extension of $q$ which is $(M,Q)$-generic.
 
If $M$ and $N$ are sets, then $\arrow{MN}$ will denote a tuple $(M,N,\epsilon)$ where $\epsilon$
is an elementary embedding
$\epsilon$ of $(M,\in)$ into $(N,\in)$ such that the range of $\epsilon$ is a countable element of $N$.
Such a tuple is will be referred to as an \emph{arrow}.
I will write $M \rightarrow N$ to mean $\arrow{MN}$ and also to indicate ``$\arrow{MN}$ is an arrow.''
If $M \rightarrow N$ and $X$ denotes an element of $M$, then $X^N$ will be used to denote $\epsilon(X)$ where
$\epsilon$ is the embedding corresponding to $M \rightarrow N$.

If $Q$ is a forcing, $M$ is suitable for $Q$, and $M \rightarrow N$,
then a filter $G \subseteq Q \cap M$ is \emph{$\arrow{MN}$-prebounded} if whenever
$N \rightarrow P$ and the image $G'$ of $G$ under the composite embeddings is in $P$,
$P$ satisfies ``$G'$ has a lower bounded.''
%If there is a need to specify $\epsilon$, we will write $M \overset{\epsilon} \rightarrow N$.
If $Q$ is a forcing notion, then a collection of embeddings
$M \rightarrow N_i$ $(i \in I)$ will be referred to as a \emph{$Q$-diagram}.
A forcing $Q$ is \emph{$\lambda$-completely proper} if whenever
$M \rightarrow N_i$ $(i \in \gamma)$ is a $Q$-diagram for $\gamma \in 1 + \lambda$,
there is a $(M,Q)$-generic filter $G \subseteq Q \cap M$ which is $\arrow{MN_i}$-prebounded
for all $i \in \gamma$.

Observe that as $\lambda$ increases, $\lambda$-complete properness becomes a weaker condition.
In \cite{minimal_unctbl_types} it was shown that $\lambda$-complete properness implies
$\Dbb$-completeness with respect to a simple $\lambda$-completeness system $\Dbb$
in the sense of Shelah (see \cite{minimal_unctbl_types} for undefined notions).
Moreover the converse is true for forcing notions $Q$ with the property that whenever
$Q_0 \subseteq Q$ is countable
$\{G \subseteq Q_0 : G \textrm{ has a lower bound in } Q\}$
is Borel.

\section{The construction}

Assume CH and fix a bijection $\ind : H(\omega_1) \to \omega_1$ such that
if $x$ is a countable subset of $\omega_1$, then $\sup(x) \in \ind (x)$.
If $x,y$ are in $H(\omega_1)$, I will abuse notation and write $\ind(x,y)$ for $\ind((x,y))$.
Fix a $\Sigma_1$-definable map $\xi \mapsto \xi^*$ which maps the countable limit ordinals into $\omega_1$ such
that if $\varrho \in \omega_1$, then $\{\xi \in \lim(\omega_1) : \xi^* = \varrho\}$ is uncountable
(for instance define $\xi^*$ to be the unique ordinal $\varrho$ such that for some $\varsigma$,
$\xi = \omega^\varsigma + \omega \cdot \varrho$ and $\omega \cdot \varrho \in \omega^\varsigma$).

If $\delta$ is a limit ordinal, let $C_\delta$ denote the cofinal subset of $\delta$ of ordertype
$\omega$ which minimizes $\ind$; set $C_{\alpha+1} = \{\alpha\}$.
Let $e_\beta : \beta \to \omega$ be defined by 
$e_\beta(\alpha) = \bar \varrho_1(\alpha,\beta)$ where $\bar \varrho_1$ is defined from
$\Seq{C_\alpha : \alpha \in \omega_1}$ as in \cite{walks}.
In what follows, we will only need that the sequence
$\Seq{e_\beta : \beta \in \omega_1}$ is determined by $\ind$, 
$e_\beta:\beta \to \omega$ is an injection, and
if $\beta \in \beta' \in \omega_1$, then
\[
\{\alpha \in \beta : e_\beta (\alpha) \ne e_{\beta'} (\alpha)\}
\]
is finite.

Let $E \subseteq \omega_1$ be a club.
Define $T = T_E = T^{\ind}_E$ to be all $t$ which are countable closed sets of limit points of $E$ such that:
\begin{itemize}

\item 
if $\nu$ is a limit point of $t$, then $t \cap \nu$ has finite intersection with
every ladder in $\nu$ which either has index less than $\ind (E \cap \nu)$ or else
is $C_\nu$;

\item
if $\nu$ is a limit point of $t$, then $\min (E \setminus \nu) \in \ind (t \cap \nu)$;

\item
for all $\alpha \in \beta$, $\ind (t \cap \beta + 1 \setminus \alpha) \in \min (t \setminus \beta + 1)$.

\end{itemize}

Define $\tilde T = \tilde T^{\ind}_E \subseteq T$ by recursion.
Begin by declaring $\emptyset \in \tilde T$.
Now suppose that we have defined $\tilde T \cap \Pscr(\xi + 1)$ for all $\xi \in \delta$.
Before defining $\tilde T \cap \Pscr(\delta +1)$, we need to specify a family of logical formulas
which will be needed in the definition.
I will use $\tilde T \restriction \nu$ to denote $\{t \in \tilde T : \ind(t) \in \nu\}$.
Let $\beta$ be a fixed ordinal less than $\delta$ and let 
$t \in \tilde T$ be such that $t \subseteq \beta$.
Consider the following recursively defined formulas
about a closed subset $x$ of $\beta$:
\begin{description}

\smallskip
\item[$\theta^\delta_0(x,t,\beta)$] \label{theta:min_cond}
$\max (t) \in \min (x)$,
$t \cup x$ is in $\tilde T \restriction \beta$, and
\[
\otp(E \cap \min (x))^* = \ind (t,n)
\]
for some $n \in \omega$;

\smallskip
\item[$\theta^\delta_1(x,t,\beta)$] \label{theta:density_cond}
if $D$ is a dense subset of $\tilde T \restriction \nu$
for some limit ordinal $\nu \in \beta$,
$\ind (D) \in \beta$, and
\[
\otp(E \cap \min (x))^* = \ind(t,e_\delta(\ind(D))),
\]
then $t \cup x$ is in $D$.

\smallskip
\item[$\theta^\delta_2(x,t,\beta)$] 
\label{theta:disjoint_cond}
if $y \subseteq \beta$, $e_\delta(\min (y)) \in e_\delta(\min (x))$ and
\[
\theta^\delta_0 \land \theta^\delta_1 \land \theta^\delta_2 \land \theta^\delta_3 (y,t,\beta),
\]
then $x \cap y \subseteq \{\min(x)\}$.

\smallskip
\item[$\theta^\delta_3(x,t,\beta)$]
if $s,z \subseteq \beta$, $\min (z) = \min (x)$, and
\[
\theta^\delta_0 \land \theta^\delta_1 \land \theta^\delta_2 (z,s,\beta),
\]
then $\ind(x) \leq \ind (z)$.

\smallskip
\end{description}
The truth of $\theta^\delta_i(x,t,\beta)$ is defined by recursion on the tuple
\[
(e_\delta( \min (x) ),\ind (x),i)
\]
equipped with the lexicographical ordering ($\delta$ is fixed).
Observe that since $e_\delta$ is injective,
there is at most one set $D$ which satisfies the hypotheses of
$\theta^\delta_1(x,t,\beta)$.
In particular, if $\theta^\delta_1(x,t,\beta)$ holds, then $\theta^\delta_1(x,t,\beta')$ holds
for all $\beta \in \beta' \in \delta$.

Now if $t$ is an element of $T$ with $\sup (t) = \delta$ and $\delta$ is not a limit point of $t$,
define $t$ to be in $\tilde T$ if and only if $t \cap \delta$ is in $\tilde T$.
If $t$ is an element of $T$ with $\sup (t) = \delta$ and $\delta$ a limit point
of $t$, then we define $t$ to be in $\tilde T$ if $t \cap \alpha + 1$ is in $\tilde T$ for
all $\alpha \in \delta$ and if, for all but finitely many consecutive pairs
$\alpha \in \beta$ in $C_\delta$ for which
$(\alpha,\beta] \cap t$ is non empty,
\[
\bigmeet_{i \in 4} \theta^\delta_i  (t \cap (\alpha,\beta],t \cap \alpha + 1,\beta).
\]

\begin{lem} \label{Trestr}
If $E$ and $E'$ are clubs such that for some $\delta \in \omega_1$,
$E' \cap \delta = E \cap \delta$, then
$\tilde T_E \cap \Pscr(\delta + 1) = \tilde T_{E'} \cap \Pscr(\delta+1)$.
\end{lem}

\begin{proof}
This follows from the observation that, under the hypotheses of the lemma,
$T_E \cap \Pscr(\delta+1) = T_{E'} \cap \Pscr(\delta+1)$ and the fact
that $\tilde T_E \cap \Pscr(\delta+1)$ is defined by recursion from
$T_E \cap \Pscr(\delta+1)$ and
$\tilde T_E \restriction \delta$.
\end{proof}

\begin{lem} \label{ad_lem}
If $t$ and $t'$ are in $\tilde T$ and $\delta \in \omega_1$ is a limit point of $\lim(t) \cap \lim(t')$,
then $t \cap \delta = t' \cap \delta$.
\end{lem}

\begin{proof}
Let $t$ and $t'$ be in $\tilde T$ and $\delta$ be a limit point of $\lim(t) \cap \lim(t')$.
It is sufficient to show that $t \cap \alpha = t' \cap \alpha$ for cofinally many $\alpha \in \delta$.
Let $\delta_0 \in \delta$ be arbitrary and let $\alpha \in \beta$ be consecutive elements of $C_\delta$
such that $\delta_ 0 \in \alpha$, $(\alpha,\beta] \cap \lim(t) \cap \lim(t')$ is non empty, and
such that
\[
\theta^\delta_i (t \cap (\alpha,\beta],t \cap \alpha + 1,\beta)
\]
\[
\theta^\delta_i (t' \cap (\alpha,\beta],t' \cap \alpha + 1,\beta)
\]
holds for all $i \in 4$.
Without loss of generality, we may assume that
\[
e_\delta(\ind(t \cap (\alpha,\beta])) \in e_\delta(\ind (t' \cap (\alpha,\beta])).
\]
Define $x = t' \cap (\alpha,\beta]$ and $y = t \cap (\alpha,\beta]$.
Observe that since $x$ and $y$ have common limit points,
$x \cap y$ is in particular not contained in $\{\min (x),\min (y)\}$.
Since $\theta^\delta_2(x,t \cap \alpha,\beta)$ is true,
it follows that $\min (x) = \min (y)$ and hence,
by $\theta^\delta_0(y,t \cap \alpha+1,\beta)$ and
$\theta^\delta_0(x,t' \cap \alpha+1,\beta)$, that
$t \cap \alpha = t' \cap \alpha$.
\end{proof}

\begin{lem}
The tree $\{(s,t) \in T^2 : (\Ht_T(s) = \Ht_T(t)) \land (s \ne t)\}$
is a union of countably many antichains.
\end{lem}

\begin{proof}
Let $U$ denote those $(s,t) \in T^2$ such that $\Ht_T(s) = \Ht_T(t)$,
$s \ne t$, and $\max(s) \leq \max (t)$.
By symmetry it is sufficient to prove that $U$ is a countable union of antichains.
For each $(s,t)$ in $U$, define $\Osc(s,t)$ be the set of all $\xi \geq \min(s \symdif t)$ such that either
\begin{itemize}

\item
$\xi$ is in $s$ and
$\min(t \setminus \xi+1) \in \min (s \setminus \xi+1)$ or

\item
$\xi$ is in $t$ and
$\min(s \setminus \xi+1) \in \min (t \setminus \xi + 1)$.

\end{itemize}
Let $\osc(s,t)$ denote the ordertype of $\Osc(s,t)$, observing that Lemma \ref{ad_lem}
implies that $\osc(s,t) \in \omega^2$ for all $(s,t)$ in $U$.
If $\max (s) \in \max (t)$, define
\[
\beta (s,t) = \min \{\beta \in t : \max (s) \in \beta\}
\]
\[
n(s,t) = e_{\beta(s,t)}(\max (s))
\]
and set $n(s,t) = \omega$ if $\max(s) = \max (t)$.
Observe that if $(s,t)$ and $(s',t')$ are in $U$ and
$s < s'$ and $t < t'$, then $\Osc(s,t)$ is an initial part of $\Osc(s',t')$
and that no element of $\Osc(s,t)$ is greater than $\max (s)$.
Consequently either
$\min (s' \setminus s)$ is in $\Osc(s',t') \setminus \Osc(s,t)$ and hence
$\osc(s,t) \ne \osc(s',t')$ or else
$\beta(s,t) = \beta(s',t')$ and
$n(s,t) \ne n(s',t')$.
It follows that, for each $\xi \in \omega^2$ and $k \in \omega$
\[
\{(s,t) \in U : (\osc(s,t) = \xi) \land (n(s,t) = k)\}
\]
is an antichain.
Since $\omega^2 \times \omega$ is countable,
this finishes the proof.
\end{proof}

\begin{lem} \label{branch_lem}
Suppose that $M$ is a countable elementary submodel of $H({2^{\omega_1}}^+)$ with $\tilde T$ in
$M$
and suppose that $t_i$ $(i \in n)$ is a sequence of elements of $\tilde T$ such that there
is a club $C \subseteq \omega_1$ in $M$ such that $C \cap M \subseteq \bigcup_{i \in n} t_i$.
Then $\tilde T$ has an uncountable chain.
\end{lem}

\begin{proof}
Let $M$, $t_i$ $(i \in n)$, and $C$ be as in the statement of the lemma.
By replacing $C$ with its limit points if necessary, we may assume that
$C \cap M \subseteq \bigcup_{i \in n} \lim(t_i)$.
Without loss of generality $t_i$ $(i \in n)$ are such that if $i \ne j$, then
$t_i \cap M \ne t_j \cap M$.
By Lemma \ref{ad_lem}, there is a $\zeta \in M \cap \omega_1$ such that if $\nu$ is a limit point
of $\lim(t_i) \cap \lim(t_j) \cap M$, then $\nu \in \zeta$.
Let $i \in n$ be such that $\lim(t_i) \cap C' \cap M$ is non-empty for every club $C'$ in $M$.
Such an $i$ exists since otherwise there would exist $C_i'$ $(i \in n)$, clubs in $M$ such that
$C \cap \bigcup_{i \in n} C_i'$ is empty.

Let $N$ be a countable elementary submodel of $H(\omega_2)$ in $M$ such that $C$ and $\zeta$ are
in $N$ and $\delta = N \cap \omega_1$ is in $\lim (t_i)$.
Since $\delta$ is not in $\lim(t_j)$ for all $j \ne i$, there
is a $\zeta' \in \delta$ such that $t_j \cap \delta \subseteq \zeta'$ whenever $j \ne i$.
Let $C'$ be the set of elements of $C$ which are greater than $\zeta'$.
If $\nu$ is in $C' \cap N$, then $\nu$ is in $\lim(t_i) \setminus \lim(t_j)$ for each $j \in n$ which is different from $i$.
Define $b$ to be the set of $p \in \tilde T$ such that
$C' \cap \sup (p)$ is an infinite cofinal subset of $\lim(p)$.
Since it is definable from parameters in $N$,
$b$ is in $N$.
Furthermore, $b$ is uncountable since it is not contained in $N$ ---
it has $t_i \cap \delta + 1$ as an element.
By Lemma \ref{ad_lem}, $b$ is a chain in $\tilde T$ and hence satisfies the conclusion of the lemma.
\end{proof}

\begin{lem}
Suppose that $E \subseteq \omega_1$ is club and $\tilde T_E$ does not
contain an uncountable chain.
Then $\tilde T_E$ is completely proper.
\end{lem}

\begin{proof}
Let $\tilde T$ denote $\tilde T_E$.
Suppose that $M \rightarrow N_i$ $(i \in \omega)$ is a $\tilde T$-diagram and $t_0$ is in $\tilde T \cap M$.
Define $\delta = M \cap \omega_1$ and let
$\eta \in \omega_1$ be an upper bound for ${\omega_1}^{N_i}$ for each $i \in \omega$.
In particular, if $X \subseteq M \cap \omega_1$ is in $N_i$, then
$\ind^{N_i}(X) \in \eta$.
Let $D_n$ $(n \in \omega)$ enumerate the dense subsets of $\tilde T$ in $M$ and let
$X_n$ $(n \in \omega)$ enumerate the collection of all cofinal subsets $X$ of $\delta$ of ordertype $\omega$
which are in $N_i$ for some $i \in \omega$.
We will construct $t_n$ $(n \in \omega)$, $\alpha_n$ $(n \in \omega)$, and $\beta_n$ $(n \in \omega)$
by induction so that for all $n \in \omega$:
\begin{enumerate}
\popcounter

\item
\label{density_cond}
$t_{n+1}$ extends $t_n$ and is in $D_n \cap M$;

\item $\alpha_n \in \beta_n  \in \alpha_{n+1} \in \delta$;

\item
\label{avoid_ladder_cond}
if $i \in n$ then $t_{n+1} \setminus t_n$ contained in $\beta_n \setminus \alpha_n$ and
is disjoint from $X_i$ if $i \in n$;

\item
if $i \in n$ then
$e^{N_i}_\delta(\xi) = e_\delta(\xi)$ whenever $\xi \in \beta_n \setminus \alpha_n$;

\item
\label{consec_cond}
if $i \in n$, then there are consecutive elements $\bar \alpha \in \bar \beta$ of $C_\delta^{N_i}$
such that $\bar \alpha \in \alpha_n \in \beta_n \in \bar \beta$;
 
\item there is a limit ordinal $\nu$ such that $\max (t_{n+1}) \in \nu \in \beta_{n+1}$ and
$\otp( E \cap \min (t_{n+1} \setminus t_n))^* = \ind (t_n,e_\delta(\ind(D \restriction \nu)))$.

\item \label{theta_cond}
if $i \in n$, then
\[
N_i \models \bigmeet_{i \in 4}\theta^\delta_i (t_{n+1} \setminus t_n,t_n,\beta_{n+1}).
\]

\pushcounter
\end{enumerate}
Assuming that this can be accomplished, then define $\bar t = \bigcup_n t_n \cup \{\delta\}$.
It follows that, if $i \in \omega$ and $N_i \rightarrow \tilde N$ with $\bar t \in \tilde N$,
then $\bar t$ is in $T^{\tilde N}$.
If $\bar t$ is in $T^{\tilde N}$ then moreover we have arranged that 
\[
N_i \models \bigmeet_{i \in 4} \theta^\delta_i (\bar t \cap (\alpha,\beta],\bar t \cap \alpha+1,\beta)
\]
whenever $\alpha \in \beta$ are consecutive elements of $C_\delta^{N_i}$ such that
$\max (t_i) \in \alpha$.
In particular $\bar t$ is in $\tilde T^{\tilde N}$.
Thus $t_n$ $(n \in \omega)$ is $\arrow{MN_i}$-prebounded in $\tilde T$ for each $i \in \omega$.

Now suppose that $t_n$ is given.
Let $M'$ be a countable elementary submodel of $H({2^{\omega_1}}^+)$ such that:
\begin{itemize}

\item $M'$ is in $M$;

\item $D_n$ is in $M'$;

\item $M'$ is an increasing union of an $\in$-chain of
elementary submodels of $H({2^{\omega_1}}^+)$.

\end{itemize}
Set $\beta_{n+1} = M' \cap \omega_1$ and
let $F \subseteq M'$ be a finite set such that if
$i \in n$, then
\[
C^{N_i}_\delta \cap M' \subseteq F
\]
\[
X_i \cap M' \subseteq F
\]
\[
\{\xi \in M' \cap \omega_1 : e^{N_i}_\delta (\xi) \ne e_\delta (\xi)\} \subseteq F
\]
Fix a $M''$ which is an elementary submodel of $H({2^{\omega_1}}^+)$ such that
$F \subseteq M'' \in M'$.
Set $\zeta = \ind (D_n \cap M'')$, $\nu = M'' \cap \omega_1$, and
let $\alpha_{n+1} \in \nu$ be such that
$\max F \in \alpha_{n+1}$.
Observe that by elementarity, $\zeta$ is in $M'$ since it is definable from
$\nu$, $D_n$,  and $\ind$.
Furthermore, $\nu \leq \zeta$ since otherwise
$D_n \cap M''$ would be in $M''$.
Observe that $e^{N_i}_\delta (\zeta) = e_\delta(\zeta)$ for all $i \in n$.
Let $\xi$ be a limit point of $E$ such that $\alpha_{n+1} \in \xi \in \nu$ and
\[
\otp(E \cap \xi)^* = \ind (t_n,e_\delta(\zeta))
\]
Let $y_i$ $(i \in l)$ list the closed subsets $y$ of $\delta$ such that:
\begin{itemize}

\item for some $j \in n$, $y \subseteq (\bar \alpha,\bar \beta]$ where
$\bar \alpha \in \bar \beta$ are the consecutive elements of $C^{N_j}_\delta$ with
$\bar \alpha \leq \alpha_{n+1} \in \beta_{n+1} \leq \bar \beta$;

\item $e^{N_j}_\delta(\min y) \in e^{N_j}_\delta(\xi)$;

\item ${\displaystyle \bigmeet_{i \in 4} \theta^\delta_i (y,t_n,\bar \beta)}$.

\end{itemize}
Observe that $\theta^\delta_3$ ensures that there are only finitely many such $y$'s.
Furthermore, Lemma \ref{branch_lem} and our hypothesis implies that
$\bigcup_{i \in l} y_i$ does not contain a set of the form $C \cap M''$ where $C$ is
a club in $M''$.
Thus there is a countable elementary submodel $M'''$ of $H(\omega_2)$ in $M''$ such that
$E$, $\tilde T$, $D_n$, $\xi$, and $\ind$ are in $M'''$ and $M''' \cap \omega_1$ is not in
$\bigcup_{i \in l} y_i$.
Let $\eta$ be a limit point of $E$ which is in $M'''$ such that $\xi$
and $\max (\bigcup_{i \in l} y_i)$ are less than $\eta$.
Since $D_n$ is dense, elementarity of $M'''$ ensures that there is an extension of
$t_n \cup \{\xi,\eta\}$ which is in $D_n \cap M'''$.
Such an extension $\tilde t$ necessarily satisfies that
$\tilde t \setminus t_n$ is disjoint from $y_i \setminus \{\xi\}$ for all $i \in l$.
We have therefore arranged that
$\theta^\delta_0 \land \theta^\delta_1 \land \theta^\delta_2 (\tilde t \setminus t_n,t_n,\beta_{n+1})$.
Let $t_{n+1}$ the element of $\tilde T$ be such that
\[
\min (t_{n+1} \setminus t_n) = \xi,
\]
\[
\theta^\delta_0 \land \theta^\delta_1 \land \theta^\delta_2 (t_{n+1} \setminus t_n,t_n,\beta_{n+1}),
\]
and $\ind(t_{n+1} \setminus t_n)$ is minimized.
It follows that $\theta^\delta_3(t_{n+1}\setminus t_n,t_n,\beta_{n+1})$.
This finishes the inductive construction and, therefore, the proof.
\end{proof}

\begin{thm} \label{main_thm}
Assume CH.
Then there is a club $E$ such that $\tilde T_E$ has no uncountable branch.
In particular, CH implies the negation of CPFA.
\end{thm}

\begin{proof}
Suppose for contradiction that CH holds and $\tilde T_E$ contains an uncountable branch for every club
$E \subseteq \omega_1$.
Let $A \subseteq \omega_1$ be such that $\ind$ is in $L[A]$.
In particular, $\Rbb \subseteq L[A]$ and $\omega_1 = \omega_1^{L[A]}$.
Inductively construct clubs $E_\xi$ $(\xi \in \omega^2)$ as follows.
Since $L[A]$ satisfies $\diamondsuit$, there is a function $h:\omega_1 \to \omega_1$
such that if $E$ is a club in $L[A]$, then for some limit point $\delta$ of $E$,
there is a ladder $X \subseteq \delta$ with $X \cap E$ infinite and $\ind(X) \in h(\delta)$.
Let $E_0$ be the $<_{L[A]}$-least club in $L[A]$ such that for all $\delta$,
$h(\delta) \in \min (E_0 \setminus \delta + 1) $.
Given $E_\xi$, let $E_{\xi+1}$ be the union of the branch through $\tilde T_{E_\xi}$.
If $E_\xi$ $(\xi \in \eta)$ has been defined for $\eta$ a limit less than $\omega^2$,
define $E_\eta = \bigcap_{\xi \in \eta} E_\xi$.
Observe that $E_2$ is not in $L[A]$ since whenever $\delta$ is a limit point of $E_1$,
\[
h(\delta) \in \min(E_0 \setminus \delta+1) \in \ind (E_1 \cap \delta).
\]
and hence $E_2$ has the property that whenever $\delta$ is a limit point of $E_2$,
 $E_2 \cap \delta$ is disjoint from $X$ whenever $X$
is a ladder in $\delta$ with index less than $h(\delta)$.
Observe that if $\xi \in \eta$, then $E_\eta$ is contained in the limit points of $E_\xi$.
Furthermore if $\delta$ is a limit point of $E_\eta$, then
\[
\min (E_\xi \setminus \delta+1) \in \ind (E_\eta \cap \delta) \in \min (E_\eta \setminus \delta+1)
\]
In particular, if ${\displaystyle \delta \in \bigcap_{\xi \in \omega^2} E_\xi}$, then for all
$k \in \omega$
\[
\sup_{i \in \omega} \ind (E_{\omega \cdot k + i} \cap \delta) \in E_\xi
\]
whenever $\xi \in \omega \cdot (k+1)$.

Let $\delta$ be the least element of $\bigcap_{\xi \in \omega^2} E_\xi$.
Observe that $\Seq{E_\xi \cap \delta : \xi \in \omega^2}$ is in $L[A]$.
We will obtain a contradiction once we show that 
$\Seq{E_\xi : \xi \in \omega^2}$ is in $L[A]$.
Working in $L[A]$,
define $\nu_\alpha$ $(\alpha \in \omega_1)$ and $t_\alpha(\xi)$ $(\xi \in \omega^2 \land \alpha \in \omega_1)$
by simultaneous recursion as follows.
The sets $t_\alpha(\xi)$ will satisfy that they are
$E_\xi \cap \nu_{\alpha}$ and the ordinals $\nu_\alpha$ will each be elements of
$\bigcap_{\xi \in \omega^2} E_\xi$.
Set $\nu_0 = \delta$ and $t_0 (\xi) = E_\xi \cap \delta$.
Given $\nu_\alpha$ and $t_\alpha(\xi)$ $(\xi \in \omega^2)$,
define
\[
\nu_{\alpha+1,k} = \sup_{\xi \in \omega \cdot k} \ind (t_\alpha (\xi))
\]
\[
\nu_{\alpha+1} = \sup_{k \in \omega} \nu_{\alpha+1,k}.
\]
Next we define $t_{\alpha+1} (\xi)$ $(\xi \in \omega^2)$ by recursion on $\xi$.
$t_{\alpha+1} (0) = E_0 \cap \nu_{\alpha+1}$.
Given $t_{\alpha+1}(\xi)$, define
$t_{\alpha+1} (\xi+1)$ to be the unique element $t$ of $\tilde T_{E_\xi} \cap \Pscr(\nu_{\alpha+1}+1)$ such that 
$\sup t = \nu_{\alpha+1}$ and
$\nu_{\alpha+1,k} \in t$ for all $k$.
(here we are employing Lemmas \ref{Trestr} and \ref{ad_lem}).
%This makes sense because, as we have noted above, the definition of
%$\tilde T_{E_\xi} \cap \Pscr(\nu_{\alpha+1})$ depends only on $E_\xi \cap \nu_{\alpha+1} = t_{\alpha+1}(\xi)$.
If $\eta \in \omega^2$ is a limit ordinal, set
\[
t_{\alpha+1} (\eta)  = \bigcap_{\xi \in \eta} t_{\alpha+1} (\xi).
\]
If $\nu_\alpha$ has been defined for all $\alpha \in \beta$,
set $\nu_\beta = \sup_{\alpha \in \beta} \nu_\alpha$ and $t_\beta(\xi) = \bigcup_{\alpha \in \beta} t_\alpha(\xi)$.

It is now easily seen that $t_\alpha (\xi) = E_\xi \cap \nu_{\alpha}$ and therefore
that $E_\xi$ is in $L[A]$ for all $\xi \in \omega^2$.
This is a contradiction, however, since $E_2$ is not in $L[A]$.
\end{proof}

I will finish by remarking that if $\rbf$ is an $\omega_1$ sequence of reals, then
the sequence $\vec{T}^\rbf$ described in the introduction is defined as follows.
If $(\omega_1)^{L[\rbf]} < \omega_1$, then define $\vec{T}^{\rbf}$ to be the empty sequence.
Otherwise, let $\ind$ be the $<_{L[\rbf]}$-least bijection between
$H(\omega_1) \cap L[\rbf]$ and $\omega_1$.
Let $A \subseteq \omega_1$ be such that $L[A] = L[\rbf]$ and define 
$T^{\rbf}_\xi = T^{\ind}_{E_\xi}$ as detailed in the proof of Theorem \ref{main_thm}.

%\bibliography{../global}

\begin{thebibliography}{10}

\bibitem{FA_CH}
D. Asper\'{o}, P. Larson, and J.~Tatch Moore.
\newblock Forcing axioms and the continuum hypothesis.
\newblock preprint, October 2010.

\bibitem{weak_diamond}
K. Devlin and S. Shelah.
\newblock A weak version of $\diamondsuit$ which follows from $2^{\aleph_0} <
  2^{\aleph_1}$.
\newblock {\em Israel Journal of Math}, 29(2--3):239--247, 1978.

\bibitem{1st_ctbl_ctbly_cpt}
T. Eisworth and P. Nyikos.
\newblock First countable, countably compact spaces and the continuum
  hypothesis.
\newblock {\em Trans. Amer. Math. Soc.}, 357(11):4269--4299, 2005.

\bibitem{tutorial_NNR}
T. Eisworth and J. Tatch~Moore.
\newblock Iterated forcing and the {C}ontinuum {H}ypothesis.
\newblock To appear in Appalachian Set Theory Workshop booklet, Jan. 2010.

\bibitem{set_theory:Kunen}
K. Kunen.
\newblock {\em {A}n {I}ntroduction to {I}ndependence {P}roofs}, volume 102 of
  {\em Studies in Logic and the Foundations of Mathematics}.
\newblock North-Holland, 1983.

\bibitem{stationary_tower}
P.~B. Larson.
\newblock {\em The stationary tower}, volume~32 of {\em University Lecture
  Series}.
\newblock American Mathematical Society, Providence, RI, 2004.
\newblock Notes on a course by W. Hugh Woodin.

\bibitem{minimal_unctbl_types}
J.~Tatch Moore.
\newblock $\omega_1$ and $-\omega_1$ may be the only minimal uncountable order
  types.
\newblock {\em Michigan Math. Journal}, 55(2):437--457, 2007.

\bibitem{proper_forcing}
S. Shelah.
\newblock {\em {P}roper and {I}mproper {F}orcing}.
\newblock Springer-Verlag, Berlin, second edition, 1998.

\bibitem{problems:Shelah}
S. Shelah.
\newblock On what {I} do not understand (and have something to say). {I}.
\newblock {\em Fund. Math.}, 166(1-2):1--82, 2000.
\newblock Saharon Shelah's anniversary issue.

\bibitem{walks}
S. Todorcevic.
\newblock {\em Walks on ordinals and their characteristics}, volume 263 of {\em
  Progress in Mathematics}.
\newblock Birkh\"auser, 2007.

\end{thebibliography}
%\bibliographystyle{plain}
%\end{document}
\def\Dbar{\leavevmode\lower.6ex\hbox to 0pt{\hskip-.23ex \accent"16\hss}D}

\end{document}